\documentclass[11pt,a4paper]{article}

\usepackage[T1]{fontenc}
\usepackage[utf8]{inputenc}
\usepackage{lmodern}

\usepackage{amsmath, amsfonts, amssymb, amsthm}
\usepackage{graphicx}
\usepackage{subcaption}
\usepackage{float}
\usepackage{multirow}
\usepackage{xcolor}
\usepackage{eurosym}

\usepackage[numbers,sort&compress]{natbib}
\usepackage{hyperref}
\usepackage{cleveref}

\usepackage[margin=1.5in]{geometry}

\theoremstyle{plain}
\newtheorem{theorem}{Theorem}[section]

\newtheorem{proposition}[theorem]{Proposition}

\theoremstyle{definition}

\theoremstyle{remark}
\newtheorem{remark}[theorem]{Remark}

\newcommand{\ym}[1]{\textcolor{black}{#1}}

\title{A Heuristic Alternating Direction Method of Multipliers Framework for Distributed and Centralized Tree-Constrained Optimization:\\
	Applications to Hop-Constrained Spanning Tree Multicommodity Flow Design}

\author{
	Yacine Mokhtari\\
	Department of Mathematical Sciences\\
	New Jersey Institute of Technology (NJIT)\\
	Newark, NJ 07102, USA\\
	\texttt{yacine.mokhtari@njit.edu}
}

\date{}  

\begin{document}
	
	\maketitle
	
	\begin{abstract}
		This paper presents centralized and distributed Alternating Direction Method of Multipliers (ADMM) frameworks for solving large-scale nonconvex optimization problems with binary decision variables subject to spanning tree or rooted arborescence constraints. We address the combinatorial complexity by introducing a continuous relaxation of the binary variables and enforcing agreement through an augmented Lagrangian formulation. The algorithms alternate between solving a convex continuous subproblem and projecting onto the tree-feasible set, reducing to a Minimum Spanning Tree or Minimum Weight Rooted Arborescence problem, both solvable in polynomial time. The distributed algorithm enables agents to cooperate via local communication, enhancing scalability and robustness. We apply the framework to multicommodity flow design with hop-constrained spanning trees. Numerical experiments demonstrate that our methods yield high-quality feasible solutions in many cases, achieving near-optimal performance.
	\end{abstract}
	
	\noindent\textbf{Keywords:} ADMM; nonconvex optimization; spanning tree; rooted arborescence; distributed optimization; multicommodity flow; Gurobi.

	\section{Introduction}
	
	\subsection{The Problem}
	Let $(\mathcal{V}, \mathcal{E})$ (resp. $(\mathcal{V}, \mathcal{A})$) be an undirected (respectively directed) graph, where $\mathcal{V}$ denotes the set of vertices and $\mathcal{E} \subset \mathcal{V} \times \mathcal{V}$ (respectively $\mathcal{A} \subset \mathcal{V} \times \mathcal{V}$) denotes the set of edges (resp. arcs). Throughout this work, we denote $|\mathcal{V}| = n$ and $|\mathcal{E}| =m$ (resp. ($|\mathcal{A}| = m$)). Bold symbols consistently represent vectors or matrices; for example, if \(\boldsymbol{x} \in \mathbb{R}^m\), then $\boldsymbol{x} = (x_1,x_2,...,x_m)$.
	
	We consider the following Mixed-Integer Non-linear Programming (MINLP) problem:
	\begin{equation}
		\begin{array}{ll}
			\text{minimize} & f\left( \boldsymbol{x}, \boldsymbol{z} \right) \\[0.3em]
			\text{subject to} & g\left( \boldsymbol{x}, \boldsymbol{z} \right) \leq \boldsymbol{0}_q, \\[0.2em]
			& \left( \boldsymbol{x}, \boldsymbol{z} \right) \in \mathcal{X} \times \mathcal{Z},
		\end{array}
		\label{eq:main_problem}
	\end{equation}
	where the continuous decision variable $\boldsymbol{x} \in \mathbb{R}^m$ belongs to the set $\mathcal{X} \subseteq \mathbb{R}^m$, and $\boldsymbol{z}$ is the binary decision variable in $\mathcal{Z}$ indicating the activation status of the edges or arcs; that is, $z_{ij} = 1$ if the edge or arc $(i,j) \in \mathcal{E}$ or $\mathcal{A}$ is active, and $z_{ij} = 0$ otherwise. The number of constraints is denoted by $q$.
	
	A topological constraint is imposed on $\boldsymbol{z}$, requiring it to define either a spanning tree (in the undirected case) or a rooted arborescence (in the directed case) with a designated root node. Formally,
	\[
	\mathcal{Z} = \left\{ \boldsymbol{z} \in \{0,1\}^m : \boldsymbol{z} \ \text{induces a spanning tree or a rooted arborescence} \right\}.
	\]

	We assume that both the objective function $f : \mathcal{X} \times \tilde{\mathcal{Z}} \to \mathbb{R}$ and the constraint function $g : \mathcal{X} \times \tilde{\mathcal{Z}} \to \mathbb{R}^q$ are jointly convex in their continuous arguments:
	\[
	(\boldsymbol{v}, \boldsymbol{w}) \mapsto f(\boldsymbol{v}, \boldsymbol{w})
	\quad \text{and} \quad
	(\boldsymbol{v}, \boldsymbol{w}) \mapsto g(\boldsymbol{v}, \boldsymbol{w}),
	\]
	are convex over $\mathcal{X} \times \tilde{\mathcal{Z}}$, where $\tilde{\mathcal{Z}}=[0,1]^m$ denotes the convex hull of $\mathcal{Z}$.
	
	\subsection{Motivation and Applications}
	Problem~\eqref{eq:main_problem} addresses the design of network topologies that must satisfy both continuous operational constraints and discrete structural requirements; specifically, that the selected edges in the undirected case form a spanning tree, or, in the directed case, a rooted arborescence. In many real-world applications, additional structural and operational constraints further increase the problem's complexity, leading to exponential growth with respect to the network size.
	
	For example, enforcing a maximum diameter or hop limit on the tree possibly in conjunction with multi-commodity flow requirements is essential in telecommunication and transportation networks, where bounded path lengths ensure acceptable latency and service quality~\cite{gouveia1996multicommodity,gouveia2003network,dahl2006formulations,magnanti1981combinatorial}. In electric power systems, radial spanning-tree constraints are fundamental in distribution network reconfiguration problems, where the objective is to minimize power losses, balance loads, or improve reliability while preserving a radial topology for protection coordination and safe operation~\cite{baran1989network,shen2019distributed,nejad2020scalable-bis,lopez2023enhanced,mokhtari2025alternating,mokhtari2025distributed}. These diverse applications underscore the need for scalable algorithms capable of handling both centralized and distributed settings on directed or undirected graphs solving problem \eqref{eq:main_problem}.
	
	\subsection{Solution Methods}
	Problem~\eqref{eq:main_problem} is a MINLP, for which both exact and approximate solution strategies exist. Among the exact approaches, the most direct method is \emph{complete enumeration}, in which all feasible combinations of the discrete variables are systematically examined, the corresponding convex optimization subproblem is solved to optimality, and the configuration with the smallest objective value is retained.
	
	More sophisticated global methods include \emph{branch-and-bound}~\cite{lawler1966branch,morrison2016branch} and \emph{branch-and-cut}~\cite{stubbs1999branch,mitchell2002branch}, which guarantee identification of the global optimum. Cutting-plane techniques~\cite{gomory1958algorithm,chvatal1989cutting} iteratively solve relaxations of the original problem and add linear inequalities to progressively enforce integrality. While these methods provide theoretical guarantees, they typically have non-polynomial worst-case complexity, making them impractical for large-scale or embedded optimization, and their runtime can vary significantly across instances.
	
	Heuristic methods can quickly produce high-quality, though not necessarily optimal, solutions. Examples include the \emph{relax-and-round} approach, which solves a convex relaxation and then projects the result onto the original nonconvex set, and approaches that fix nonconvex variables to plausible values before solving the remaining convex subproblems~\cite{boyd-biconvex}. Feasibility-oriented heuristics, such as the feasibility pump~\cite{achterberg2007improving,fischetti2005feasibility}, aim to find feasible solutions efficiently. Although lacking theoretical guarantees, these methods are often effective in practice and well-suited to time-constrained settings.
	
	\subsection{The Alternating Direction Method of Multipliers as a Heuristic}
	
	The Alternating Direction Method of Multipliers (ADMM) is a primal--dual splitting algorithm originally developed for convex optimization problems~\cite{Boyd1,Bertsekas:book16:NLProg}.
	
	\ym{Over the past decade, ADMM has gained significant popularity as a practical heuristic for nonconvex and nonsmooth optimization, frequently delivering high-quality approximate solutions to NP-hard problems despite the absence of global optimality guarantees~\cite{boyd2011distributed,hong2016convergence,wang2019global}. Its use as a general heuristic for nonconvex optimization has been explored in, e.g., \cite[Ch.~9]{Boyd1} and~\cite{derbinsky2013improved}}.
	
	\ym{Theoretical progress has established convergence to stationary points for various nonconvex problem classes under mild assumptions, such as sufficiently large penalty parameters or the Kurdyka--Łojasiewicz property~\cite{hong2016convergence,wang2019global,yuan2024admm}. Recent extensions include inexact and stochastic variants for nonsmooth nonconvex objectives~\cite{bai2024inexact} as well as specialized heuristics for nonconvex conic constraints~\cite{alzalg2025admm}. However, ensuring convergence in the general nonconvex case remains an open challenge.}
	
	ADMM can address computationally challenging mixed-integer nonlinear programs (MINLPs)~\cite{boyd-biconvex} when projection onto the discrete constraint set can be performed, either exactly or approximately. It has been successfully applied in various mixed-integer optimization contexts, including mixed-integer quadratic programming~\cite{takapoui2020simple}, pump scheduling and water network management~\cite{fooladivanda2017energy}, weighted network design with cardinality constraints~\cite{sun2018weighted}, electric distribution system reconfiguration~\cite{shen2019distributed,mokhtari2025alternating,mokhtari2025distributed}.
	
	\ym{ADMM's primary advantage is its decomposition of large-scale problems into sequences of simpler subproblems often admitting closed-form or efficient solutions while maintaining global coordination through penalty terms and dual updates. This makes ADMM highly scalable for centralized optimization and a natural framework for distributed algorithms~\cite{Boyd1,yang2019survey,molzahn2017survey,forero2010consensus,huang2016consensus}.}

	\subsection{Distributed Optimization}
	Distributed optimization encompasses a class of algorithms in which multiple agents cooperate to solve a global optimization problem. Such methods are valued for their \emph{scalability}, \emph{robustness}, \emph{privacy preservation}, and \emph{adaptability}. By partitioning the computational workload among agents, these algorithms are well-suited to large-scale systems, can tolerate local faults or communication losses, and alleviate centralized bottlenecks. For privacy-sensitive scenarios, distributed schemes enable local data processing without sharing raw information, thereby reducing both privacy risks and communication overhead compared to fully centralized architectures~\cite{molzahn2017survey,yang2019survey}.
	
	Applications span a wide range of domains, including networked multi-agent coordination, large-scale machine learning (e.g., distributed training of Support Vector Machines), and smart grid management~\cite{lin2003multi,cortes1995support,molzahn2017survey,nabavi2015distributed,mokhtari2025distributed}. A canonical formulation involves $N$ agents jointly solving
	\begin{equation}
		\underset{x \in \mathbb{R}^{N}}{\min} \ \sum_{i=1}^{N} f_{i}(x),
		\label{consensus}
	\end{equation}
	where $f_i : \mathbb{R}^N \to \mathbb{R}$ denotes the local objective of agent $i$ and $x$ is a global decision variable shared by all agents.
	
	Although distributed optimization is well-developed for convex and continuous problems~\cite{boyd2011distributed}, many practical settings, such as network reconfiguration, facility location, and scheduling, necessitate mixed-integer formulations with binary or general discrete variables. These distributed MINLPs are substantially more difficult due to their combinatorial nature, non-convex feasible regions, and the absence of scalable, exact distributed solvers~\cite{camisa2021distributed,camisa2022distributed}. In such contexts, enforcing global combinatorial constraints (e.g., radiality in network topologies) while relying only on local agent knowledge typically demands consensus-based or primal–dual decomposition methods, augmented with advanced projection or relaxation strategies. 
	
	\subsection{Related Works and Contribution}
	In the context of heuristic methods for optimization problems with tree constraints, several works have proposed distributed ADMM-based approaches to address the power flow network reconfiguration problem under radiality constraints~\cite{shen2019distributed,nejad2020scalable-bis,lopez2023enhanced}. However, these methods do not guarantee that the final solution satisfies the tree constraint. The main limitation lies in the projection step, which often relies on rounding of relaxed continuous variables to binary decisions, without explicitly enforcing the radiality condition. As a result, the obtained topology may contain cycles or be disconnected, violating feasibility.
	
	This limitation has been addressed in subsequent works in both centralized~\cite{mokhtari2025alternating} and distributed~\cite{mokhtari2025distributed} settings, where the projection step is reformulated as an exact combinatorial optimization problem. Specifically, it is treated as a \ym{Minimum Spanning Tree (MST)} problem in the undirected case or a \ym{Minimum Weight Rooted Arborescence (MWRA)} problem in the directed case~\cite[Chapter~6.3]{korte2018combinatorial}. This guarantees that the resulting topology is always a feasible tree structure, thereby ensuring strict satisfaction of the radiality constraint.
	
	In this work, we extend prior ADMM-based approaches~\cite{shen2019distributed,nejad2020scalable-bis,lopez2023enhanced} by developing centralized and distributed algorithms that ensure tree constraints via exact MST or MWRA projections, applicable to directed and undirected graphs. As a case study, we demonstrate the effectiveness of both algorithms on multicommodity flow formulations for spanning trees subject to hop constraints~\cite{gouveia1996multicommodity,gouveia2003network,dahl2006formulations,magnanti1981combinatorial}.

	\section{The Algorithm}
	
	\subsection{Review}
	Before presenting the proposed algorithms, we recall two fundamental combinatorial optimization problems:
	\subsubsection{Minimum Spanning Tree (MST)}
	\label{Subsec:MST}
	
	The MST problem arises in undirected graphs and provides the basis for enforcing tree structures in $\mathcal{Z}$ for undirected network design.  
	Let $G = (\mathcal{V}, \mathcal{E})$ be an undirected graph, and let $\{h_{ij}\}_{(i,j) \in \mathcal{E}}$ denote the set of edge weights.  
	The MST problem seeks a subset of edges $\mathcal{T} \subset \mathcal{E}$ such that $(\mathcal{V}, \mathcal{T})$ forms a \emph{spanning tree}, i.e., a connected and acyclic subgraph that includes all vertices, while minimizing the total weight:
	\begin{equation}
		\sum_{(i,j) \in \mathcal{T}} h_{ij}.
		\label{WEIGHTS_MST}
	\end{equation}
	The MST problem can be solved in polynomial time via classical algorithms such as Kruskal's or Prim's method, with respective time complexities $O(m \log n)$ and $O(m + n \log n)$ \cite[Chapter~6.3]{korte2018combinatorial}.
	
	\subsubsection{Minimum Weight Rooted Arborescence (MWRA)}
	\label{Subsec:MWRA}
	
	The MWRA problem generalizes the MST concept to directed graphs.  
	Given a directed graph $G = (\mathcal{V}, \mathcal{A})$, a designated root node $r \in \mathcal{V}$, and a set of arc weights $\{h_{ij}\}_{(i,j) \in \mathcal{A}}$, the goal is to identify a directed spanning tree (arborescence) $\mathcal{T} \subset \mathcal{A}$ rooted at $r$ such that there exists a \emph{unique} directed path from $r$ to every other node, while minimizing the total weight:
	\begin{equation}
		\sum_{(i,j) \in \mathcal{T}} h_{ij}.
		\label{WEIGHTS_MWRA}
	\end{equation}
	This problem can be solved in polynomial time using Edmonds' algorithm, with complexity $O(mn)$ \cite[Chapter~6.3]{korte2018combinatorial}.

	\subsection{The centralized ADMM Algorithm}
	To address the combinatorial nature of the problem, we introduce a
	continuous relaxation variable $\boldsymbol{w} \in [0,1]^m $ to serve as a
	surrogate for the binary variable $\boldsymbol{z} $. This yields the
	following relaxed formulation: 
	\begin{equation*}
		\begin{array}{ll}
			\text{minimize} & f(\boldsymbol{x}, \boldsymbol{w}) \\ 
			\text{subject to} & g(\boldsymbol{x}, \boldsymbol{w}) \leq \boldsymbol{0}_q, \\ 
			& \boldsymbol{w} = \boldsymbol{z}, \\ 
			& (\boldsymbol{x}, \boldsymbol{w}, \boldsymbol{z}) \in \mathcal{X} \times
			[0,1]^m \times \mathcal{Z}.%
		\end{array}%
	\end{equation*}
	The corresponding augmented Lagrangian function is defined as: 
	\begin{equation*}
		\mathcal{L}_{\rho }(\boldsymbol{x},\boldsymbol{w},\boldsymbol{z},\boldsymbol{%
			\mu })=f(\boldsymbol{x},\boldsymbol{w})+\boldsymbol{\mu }^{\top }(%
		\boldsymbol{z}-\boldsymbol{w})+\frac{\rho }{2}\left\Vert \boldsymbol{z}-%
		\boldsymbol{w}\right\Vert^{2},
	\end{equation*}%
	where $\boldsymbol{\mu }\in \mathbb{R}^{m}$ is the vector of Lagrange
	multipliers, and $\rho >0$ is a penalty parameter.
	
	We define the set 
	\begin{equation*}
		\Sigma =\left\{ (\boldsymbol{x},\boldsymbol{w})\in \mathcal{X}\times \lbrack
		0,1]^{m}:g(\boldsymbol{x},\boldsymbol{w})\leq \boldsymbol{0}_q\right\} .
	\end{equation*}
	and we assume that it is nonempty. The ADMM update steps then read: 
	\begin{subequations}
		\label{eq:admm_updates}
		\begin{align}
			(\boldsymbol{x},\boldsymbol{w})_{k+1}& =\underset{(\boldsymbol{x},%
				\boldsymbol{w})\in \Sigma }{\arg \min }\left[ f(\boldsymbol{x},\boldsymbol{w}%
			)+\frac{\rho }{2}\Vert \boldsymbol{z}_{k}-\boldsymbol{w}+\boldsymbol{\mu }%
			_{k}\Vert^{2}\right] ,  \label{AA1} \\[1ex]
			\boldsymbol{z}_{k+1}& =\underset{\boldsymbol{z}\in \mathcal{Z}}{\arg \min }%
			\Vert \boldsymbol{z}-\boldsymbol{w}_{k+1}+\boldsymbol{\mu }_{k}\Vert^{2},  \label{AA2} \\[1ex]
			\boldsymbol{\mu }_{k+1}& =\boldsymbol{\mu }_{k}+\boldsymbol{z}_{k+1}-%
			\boldsymbol{w}_{k+1}.  \label{AA3}
		\end{align}
	\end{subequations}
	
	The subproblem~\eqref{AA1} is a continuous convex optimization problem that can be efficiently solved using standard convex optimization solvers. In contrast, problem~\eqref{AA2} is a MINLP, specifically an integer quadratic programming problem, whose solution enforces a spanning tree or rooted arborescence structure. In this case, the step reduces to solving an MST problem or MWRA problem with iteratively updated edge weights.

	\begin{proposition}
		\label{Proposition} Problem~\eqref{AA2} is equivalent to solving the
		following discrete optimization problem at each iteration: 
		\begin{equation}
			\boldsymbol{z}_{k+1}=\underset{\boldsymbol{z}\in \mathcal{Z}}{\arg \min }\;%
			\boldsymbol{z}^{T}\boldsymbol{h}_{k},  \label{weight}
		\end{equation}%
		where the weight vector $\boldsymbol{h}_{k}\in \mathbb{R}^{m}$ is given by 
		\begin{equation}
			\boldsymbol{h}_{k}=\boldsymbol{\mu }_{k}-\boldsymbol{w}_{k+1}.  \label{h_k}
		\end{equation}%
	\end{proposition}
	
	\begin{proof}
		\label{Proof_prop} We begin by expanding the squared norm appearing in the
		second ADMM subproblem: We aim to show that problem~\eqref{AA2} is
		equivalent to solving the discrete optimization problem~\eqref{weight}.
		Starting with the objective in~\eqref{AA2}, we expand the squared norm: 
		\begin{equation*}
			\left\Vert \boldsymbol{z}-\boldsymbol{w}_{k+1}+\boldsymbol{\mu }%
			_{k}\right\Vert ^{2}=\boldsymbol{z}^{\top }\boldsymbol{z}+\left\Vert 
			\boldsymbol{\mu }_{k}-\boldsymbol{w}_{k+1}\right\Vert ^{2}+2\boldsymbol{z}%
			^{\top }\left( \boldsymbol{\mu }_{k}-\boldsymbol{w}_{k+1}\right) .
		\end{equation*}%
		Since $\boldsymbol{z}\in \mathcal{Z}\subset \{0,1\}^{m}$ represents a
		spanning tree or arborescence with exactly $n-1$ edges, we have $\boldsymbol{%
			z}^{\top }\boldsymbol{z}=\sum_{(i,j)\in \mathcal{E}\text{ or }\mathcal{A}%
		}z_{ij}=n-1$, as each selected edge contributes 1 to the sum. Thus, the
		expression becomes: 
		\begin{eqnarray*}
			\left\Vert \boldsymbol{z}-\boldsymbol{w}_{k+1}+\boldsymbol{\mu }%
			_{k}\right\Vert ^{2} &=&(n-1)+\left\Vert \boldsymbol{\mu }_{k}-\boldsymbol{w}%
			_{k+1}\right\Vert ^{2}+2\boldsymbol{z}^{\top }\left( \boldsymbol{\mu }_{k}-%
			\boldsymbol{w}_{k+1}\right)  \\
			&=&(n-1)+\left\Vert \boldsymbol{\mu }_{k}-\boldsymbol{w}_{k+1}\right\Vert
			^{2}+2\boldsymbol{z}^{\top }\boldsymbol{h}_{k}.
		\end{eqnarray*}%
		Since $n-1$ and $\left\Vert \boldsymbol{\mu }_{k}-\boldsymbol{w}%
		_{k+1}\right\Vert ^{2}$ are constant with respect to $\boldsymbol{z}$,
		minimizing the expression over $\boldsymbol{z}\in \mathcal{Z}$ is equivalent
		to minimizing $\boldsymbol{z}^{\top }\boldsymbol{h}_{k}$. Thus, problem~%
		\eqref{AA2} reduces to: 
		\begin{equation*}
			\boldsymbol{z}_{k+1}=\underset{\boldsymbol{z}\in \mathcal{Z}}{\arg \min }\ 
			\boldsymbol{z}^{\top }\boldsymbol{h}_{k},
		\end{equation*}%
		which matches~\eqref{weight}.
	\end{proof}

	Problem~\eqref{weight} is an integer linear optimization problem. When the set $\mathcal{Z}$ corresponds to the collection of indicator vectors of spanning trees in an undirected graph, it reduces to the MST problem~\eqref{WEIGHTS_MST}. Conversely, if $\mathcal{Z}$ represents the set of rooted arborescences in a directed graph, it becomes the MWRA problem~\eqref{WEIGHTS_MWRA}. In both formulations, the edge or arc weights are determined by the components of the vector $\boldsymbol{h}_k$ at each iteration $k \geq 1$.
	
	\subsection{The Distributed ADMM Algorithm}
	
	We assume that the objective function $f $ and the constraint function $g$
	are separable, i.e.,
	
	\begin{equation*}
		f(\boldsymbol{x}, \boldsymbol{z}) = \sum_{i \in \mathcal{V}} f^i(\boldsymbol{x}, \boldsymbol{z}), 
		\quad 
		\sum_{i \in \mathcal{V}} g^i(\boldsymbol{x}, \boldsymbol{z}) \leq \boldsymbol{0}_q,
	\end{equation*}
	where $f^i$ and $g^i$ denote, respectively, the local objective function and the local constraint vector associated with agent $i \in \mathcal{V}$. The distributed formulation of problem~\eqref{distributed} can thus be expressed as:
	
	\begin{equation}
		\begin{array}{ll}
			\text{minimize} & \sum_{i \in \mathcal{V}} f^i(\boldsymbol{x}, \boldsymbol{z}
			) \\ 
			\text{subject to} & \sum_{i \in \mathcal{V}} g^i(\boldsymbol{x}, \boldsymbol{z}) \leq \boldsymbol{0}_q, \\ 
			& (\boldsymbol{x}, \boldsymbol{z}) \in \mathcal{X} \times 
			\mathcal{Z}.%
		\end{array}
		\label{distributed}
	\end{equation}
	Consider a directed graph $G=(\mathcal{V},\mathcal{A})$. Our objective is to
	establish a distributed method that enables the agents to cooperatively solve
	the optimization problem~\eqref{distributed} while exchanging information
	only with their neighbors in $G$. 
	
	Fix any $\rho > 0$ and initialize, for each agent $i \in \mathcal{V}$, the dual states as $\boldsymbol{\nu}^{\,i}_0 = \boldsymbol{\xi}^{\,i}_0 = \boldsymbol{0}_m$. Then, for every iteration $k \in \mathbb{N}$, each agent $i \in \mathcal{V}$ performs the following update steps:

	\begin{subequations}
		\begin{eqnarray}
			(\boldsymbol{x}^{i},\boldsymbol{w}^{i})_{k+1} &=&\underset{(\boldsymbol{x}%
				^{i},\boldsymbol{w}^{i})\in \Sigma ^{i}}{\arg \min }\left\{ f^{i}(%
			\boldsymbol{x}^{i},\boldsymbol{w}^{i})+\frac{\rho }{2}\left\Vert \boldsymbol{%
				z}^{i}_k-\boldsymbol{w}^{i}+\boldsymbol{\mu }_{k}^{i}\right\Vert^{2}+(\boldsymbol{\nu }_{k}^{i})^{\top }\boldsymbol{x}^{i}+(\boldsymbol{%
				\xi }_{k}^{i})^{\top }\boldsymbol{w}^{i}\right.\label{distr1}  \\
			&&\left. +\frac{\rho }{2}\sum_{j\in N^{-}(i)\cup N^{+}(i)}\left\Vert 
			\boldsymbol{x}^{i}-\frac{\boldsymbol{x}_{k}^{i}+\boldsymbol{x}_{k}^{j}}{2}%
			\right\Vert ^{2}+\frac{\rho }{2}\sum_{j\in N^{-}(i)\cup
				N^{+}(i)}\left\Vert \boldsymbol{w}^{i}-\frac{\boldsymbol{w}_{k}^{i}+%
				\boldsymbol{w}_{k}^{j}}{2}\right\Vert ^{2}\right\} , \label{distributed1} \notag
			\\
			\boldsymbol{z}_{k+1}^{i} &=&\underset{\boldsymbol{z}^{i}\in \mathcal{Z}}{%
				\arg \min }\left\Vert \boldsymbol{z}^{i}-\boldsymbol{w}_{k+1}^{i}+%
			\boldsymbol{\mu }_{k}^{i}\right\Vert ^{2}, \label{Z1}\\
			\boldsymbol{\mu }_{k+1}^{i} &=&\boldsymbol{\mu }_{k}^{i}+\boldsymbol{z}^{i}_{k+1}-%
			\boldsymbol{w}^{i}_{k+1}, \\
			\boldsymbol{\nu }_{k+1}^{i} &=&\boldsymbol{\nu }_{k}^{i}+\frac{1 }{2}%
			\sum_{j\in N^{-}(i)\cup N^{+}(i)}\left( \boldsymbol{x}_{k+1}^{i}-\boldsymbol{%
				x}_{k+1}^{j}\right) , \\
			\boldsymbol{\xi }_{k+1}^{i} &=&\boldsymbol{\xi }_{k}^{i}+\frac{1}{2}%
			\sum_{j\in N^{-}(i)\cup N^{+}(i)}\left( \boldsymbol{w}_{k+1}^{i}-\boldsymbol{%
				w}_{k+1}^{j}\right) .\label{distributed4}
		\end{eqnarray}
	\end{subequations}
	where the local constraint set $\Sigma^i$, for $i\in \mathcal{V}$ as
	\begin{equation}
		\Sigma^i =\left\{ (\boldsymbol{x},\boldsymbol{w})\in \mathcal{X}\times \lbrack
		0,1]^{m}:g^i(\boldsymbol{x},\boldsymbol{w})\leq \boldsymbol{0}_q\right\},\label{local-constraints}
	\end{equation}
	and \(N^{-}(i)\) and \(N^{+}(i)\) denote the sets of in-neighbors and out-neighbors of node \(i \in \mathcal{V}\), respectively.

	For the derivation of the above algorithm see Appendix \ref{Appendix}. Note that Problem~\eqref{distr1} is a convex optimization problem that can be handled with optimization solvers. Similar to the centralized case, the subproblem \eqref{Z1} reduces to an MWRA problem. We have	
	
	\begin{proposition}
		\label{Proposition_dis} For all agents $i\in \mathcal{V}$, problem~\eqref{Z1} is equivalent to solving the
		following discrete optimization problem at each iteration $k\geq0$:
		\begin{equation}
			\boldsymbol{z}^i_{k+1}=\underset{\boldsymbol{z}\in \mathcal{Z}}{\arg \min }\;%
			\boldsymbol{z}^{T}\boldsymbol{h}^i_{k},  \label{weight_d}
		\end{equation}%
		where the weight vector $\boldsymbol{h}^i_{k}\in \mathbb{R}^{m}$ is given by 
		\begin{equation*}
			\boldsymbol{h}_{k}^{i}=\boldsymbol{\mu }_{k}^{i}-
			\boldsymbol{w}_{k+1}^{i},\text{ }i\in \mathcal{V}.
		\end{equation*}
		
	\end{proposition}
	
	\begin{proof}
		The proof is similar to Proof \ref{Proof_prop}.
	\end{proof}
	\begin{remark}
		When the graph is undirected, the above algorithm still works by observing
		that $N^{-}(i)=N^{+}(i)$, for all $i\in \mathcal{V}$. It takes the following
		form: 
		
		\begin{subequations}
			\begin{eqnarray}
				(\boldsymbol{x}^{i},\boldsymbol{w}^{i})_{k+1} &=&\underset{(\boldsymbol{x}%
					^{i},\boldsymbol{w}^{i})\in \Sigma ^{i}}{\arg \min }\left\{ f^{i}(%
				\boldsymbol{x}^{i},\boldsymbol{w}^{i})+\frac{\rho }{2}\left\Vert \boldsymbol{%
					z}^{i}_k-\boldsymbol{w}^{i}+\boldsymbol{\mu }_{k}^{i}\right\Vert^{2}+(\boldsymbol{\nu }_{k}^{i})^{\top }\boldsymbol{x}^{i}+(\boldsymbol{%
					\xi }_{k}^{i})^{\top }\boldsymbol{w}^{i}\right.  \\
				&&+\left. \rho \sum_{j\in N(i)}\left\Vert \boldsymbol{x}^{i}-\frac{%
					\boldsymbol{x}_{k}^{i}+\boldsymbol{x}_{k}^{j}}{2}\right\Vert^{2}+\rho \sum_{j\in N(i)}\left\Vert \boldsymbol{w}^{i}-\frac{%
					\boldsymbol{w}_{k}^{i}+\boldsymbol{w}_{k}^{j}}{2}\right\Vert^{2}\right\} ,  \notag \\
				\boldsymbol{z}_{k+1}^{i} &=&\underset{\boldsymbol{z}^{i}\in \mathcal{Z}}{%
					\arg \min }\left\Vert \boldsymbol{z}^{i}-\boldsymbol{w}_{k+1}^{i}+%
				\boldsymbol{\mu }_{k}^{i}\right\Vert ^{2}, \label{Z2}\\
				\boldsymbol{\mu }_{k+1}^{i} &=&\boldsymbol{\mu }_{k}^{i}+\boldsymbol{z}^{i}_{k+1}-%
				\boldsymbol{w}^{i}_{k+1}, \\
				\boldsymbol{\nu }_{k+1}^{i} &=&\boldsymbol{\nu }_{k}^{i}+\sum_{j\in
					N(i)}\left( \boldsymbol{x}_{k+1}^{i}-\boldsymbol{x}_{k+1}^{j}\right) , \\
				\boldsymbol{\xi }_{k+1}^{i} &=&\boldsymbol{\xi }_{k}^{i}+ \sum_{j\in
					N(i)}\left( \boldsymbol{w}_{k+1}^{i}-\boldsymbol{w}_{k+1}^{j}\right) .
			\end{eqnarray}
		\end{subequations}
		
	\end{remark}
	
	\subsection{Execution Details}
	
	\subsubsection{Initialization}
	As noted in \cite{boyd-biconvex,takapoui2020simple}, ADMM on nonconvex models is sensitive to initialization and to the penalty parameter $\rho$. We adopt a relaxed start by sampling $\mathbf{w}_0$ from the convex hull of $\mathcal{Z}$. In the distributed case, we initialize all agents identically (e.g., $\boldsymbol{w}_0^{\,i}=\bar{\boldsymbol{w}}$, $\boldsymbol{x}_0^{\,i}=\bar{\boldsymbol{x}}$ for all $i$) to reduce transient disagreement and to speed up the convergence. Nonetheless, convergence and consensus ultimately depend on the update rules and graph connectivity \cite{shi2014linear}. The penalty parameter $\rho$ trades off feasibility and optimality: a larger $\rho$ enforces constraints more strictly, while a smaller $\rho$ favors objective improvement. 
	
	\subsubsection{Computational Cost}
	In the centralized setting, in each iteration of ADMM, the main computational burden lies in solving the convex subproblem~\eqref{AA1}. By contrast, the projection step~\eqref{AA2} onto the tree constraint is significantly more efficient, which can be solved exactly in a polynomial time in both directed and undirected graphs. In the distributed setting, the same complexity occurs for each agent $i\in \mathcal{V}$. 
	\subsubsection{Convergence}
	Because $\mathcal{Z}$ is a nonconvex constraint set, global convergence of ADMM is not guaranteed in general \cite{boyd-biconvex}. In practice, however, we observe that the method yields high-quality approximate solutions for our application.
	
	To monitor progress in the \emph{centralized} setting, we define the total residual at iteration $k\ge 1$ as
	\begin{equation}
		\bigl\|\boldsymbol{\mu}_{k}-\boldsymbol{\mu}_{k-1}\bigr\|
		\;+\;
		\bigl\|(\boldsymbol{x},\boldsymbol{w})_{k}-(\boldsymbol{x},\boldsymbol{w})_{k-1}\bigr\|,
		\label{error_c}
	\end{equation}
	which aggregates dual and primal changes. Algorithm~\eqref{AA1}--\eqref{AA3} is declared converged when the residual is below the tolerance parameter.
	
	In the distributed setting, we define the residual
	\begin{equation}
		\frac{1}{n}\sum_{i\in\mathcal{V}}
		\Bigl(
		\bigl\|(\boldsymbol{\mu}^i,\boldsymbol{\nu}^i,\boldsymbol{\xi}^i)_{k}-(\boldsymbol{\mu}^i,\boldsymbol{\nu}^i,\boldsymbol{\xi}^i)_{k-1}\bigr\|
		+
		\bigl\|(\boldsymbol{x}^i,\boldsymbol{w}^i)_{k}-(\boldsymbol{x}^i,\boldsymbol{w}^i)_{k-1}\bigr\|
		\Bigr),
		\label{error_d}
	\end{equation}
	and stop when the "average" residual \eqref{error_d} is below the tolerance error parameter.

	\begin{remark}
		For clarity of presentation, we adopt the centralized stopping criterion 
		\eqref{error_d}, which contrasts with the inherently distributed nature of 
		Algorithm~\eqref{distributed1}--\eqref{distributed4}. The choice is motivated by the 
		fact that a fully distributed stopping rule may lead each agent to terminate 
		at different iterations, making the results less straightforward to interpret numerically. 
		Nevertheless, the proposed framework can be adapted to incorporate a distributed 
		stopping criterion, as discussed in \cite{asefi2020distributed}.
	\end{remark}

	\section{Application: Multicommodity Flow Formulations for Spanning Trees with Hop Constraints}
	
	\subsection{The Model}
	Let \(G=(\mathcal{V},\mathcal{E})\) be an undirected graph with $|\mathcal{E}|=m$, where each edge \((i,j)\in\mathcal{E}\) has a nonnegative cost \(c_{ij}\in\mathbb{R}_{\ge 0}\).
	For flow modeling, we work with the bidirected arc set
	\[
	\mathcal{A} \;=\; \{(i,j),\,(j,i)\,:\,(i,j)\in\mathcal{E}\}.
	\]
	
	We define a set of commodities \(F\); each commodity \(f\in F\) has an origin \(O(f)\in\mathcal{V}\) and a destination \(D(f)\in\mathcal{V}\).
	The goal is to find a minimum-cost spanning tree (on \(\mathcal{E}\)) that supports routing all commodities (over \(\mathcal{A}\)) while ensuring that each commodity’s path length does not exceed a bound \(d\in\mathbb{N}\)
	(see e.g., \cite{gouveia1996multicommodity,gouveia2003network,dahl2006formulations,magnanti1981combinatorial}).
	
	The binary variable \(z_{ij}\in\{0,1\}\) indicates whether undirected edge $(i,j)\in\mathcal{E}$ is in the tree (we write \(z_{ij}=z_{ji}\) for $(i,j)\in\mathcal{E}$).
	For each commodity \(f\in F\), the arc-flow variables \(y_{ij}^f\in\{0,1\}\) indicate whether commodity \(f\) uses arc \((i,j)\in\mathcal{A}\).
	We let \(\boldsymbol{y}^f\in\{0,1\}^{2m}\) collect all flows for \(f\).
	The model reads \cite{gouveia1996multicommodity,gouveia2003network,dahl2006formulations}:
	\begin{subequations}\label{eq:MCF}
		\begin{align}
			& \underset{(\boldsymbol{y},\boldsymbol{z})}{\text{minimize}}
			&& \sum_{(i,j)\in\mathcal{E}} c_{ij}\, z_{ij}, \label{eq:MCF_obj}\\[0.3em]
			& \text{subject to}
			&& \sum_{(i,j)\in\mathcal{E}} z_{ij} = n-1,\quad \sum_{\mathcal{E}(S)}z_{ij}\leq|S|-1,\ S\subset\mathcal{V}, \ |S|\geq2, \label{eq:MCF_tree_size}\\[0.3em]
			&&& \sum_{j\in N^{-}(i)} y_{ji}^{f} - \sum_{j\in N^{+}(i)} y_{ij}^{f} \;=\;
			\begin{cases}
				1  & \text{if } i=O(f),\\
				0  & \text{if } i\notin\{O(f),D(f)\},\\
				-1 & \text{if } i=D(f),
			\end{cases}
			\quad \forall i\in\mathcal{V},\ \forall f\in F, \label{eq:MCF_flow_conservation}\\[0.5em]
			&&& y_{ij}^{f} + y_{ji}^{f} \;\le\; z_{ij},
			\quad \forall (i,j)\in\mathcal{E},\ \forall f\in F, \label{eq:MCF_edge_activation}\\[0.3em]
			&&& \sum_{(i,j)\in\mathcal{A}} y_{ij}^{f} \;\le\; d,
			\quad \forall f\in F, \label{eq:MCF_diameter}\\[0.3em]
			&&& y_{ij}^{f}\in\{0,1\},\ \forall (i,j)\in\mathcal{A},\ \forall f\in F,\qquad
			z_{ij}\in\{0,1\},\ \forall (i,j)\in\mathcal{E},
			\label{eq:MCF_binary}
		\end{align}
	\end{subequations}
	where $\mathcal{E}(S)$ denotes the set of edges whose endpoints both lie in $S$. 
	Constraint~\eqref{eq:MCF_tree_size} ensures that the selected edges 
	$\boldsymbol{z}$ form a connected spanning tree over $\mathcal{E}$; 
	see, for example, \cite{myung1995generalized,abdelmaguid2018efficient}.
	
	\subsection{The Centralized Algorithm}
	Next, we define a relaxed version of the model, where \( w_{ij} \in [0,1] \) and \( u_{ij}^f \in [0,1] \) replace \( z_{ij} \) and \( y_{ij}^f \), respectively:
	\begin{subequations}
		\begin{align}
			& \underset{(\boldsymbol{u}, \boldsymbol{w})}{\text{minimize}}
			&& \sum_{(i,j) \in \mathcal{E}} c_{ij} \, w_{ij},
			\label{eq:objective2} \\[0.3em]
			& \text{subject to}
			&& \sum_{(i,j) \in \mathcal{E}} w_{ij} = n - 1,\quad \sum_{\mathcal{E}(S)}w_{ij}\leq|S|-1,\ S\subset\mathcal{V}, \ |S|\geq2,
			\label{eq:constraint1} \\[0.3em]
			&&& \boldsymbol{w} = \boldsymbol{z}, \quad \boldsymbol{u}^f = \boldsymbol{y}^f,
			\label{eq:constraint2} \\[0.3em]
			&&& \sum_{j \in N^{-}(i)} u_{ji}^{f} - \sum_{j \in N^{+}(i)} u_{ij}^{f} =
			\begin{cases}
				1 & \text{if } i = O(f), \\[0.2em]
				0 & \text{if } i \notin \{O(f), D(f)\}, \\[0.2em]
				-1 & \text{if } i = D(f),
			\end{cases}
			\quad \forall i \in \mathcal{V},\ \forall f \in F,
			\label{eq:constraint3} \\[0.5em]
			&&& u_{ij}^f + u_{ji}^f \leq w_{ij},
			\quad \forall (i,j) \in \mathcal{E},\ \forall f \in F,
			\label{eq:constraint4} \\[0.3em]
			&&& \sum_{(i,j) \in \mathcal{A}} u_{ij}^f \leq d,
			\quad \forall f \in F,
			\label{eq:constraint5} \\[0.3em]
			&&& w_{ij} \in [0,1],\quad \forall (i,j) \in \mathcal{E}, \nonumber\\
			&&& u_{ij}^f \in [0,1],\quad \forall (i,j) \in \mathcal{E},\ \forall f \in F.
			\label{eq:constraint6}
		\end{align}
	\end{subequations}
	The ADMM updates are defined as follows:
	\begin{subequations}
		\begin{align}
			(\boldsymbol{u}, \boldsymbol{w})_{k+1}
			&= \arg\min_{(\boldsymbol{u}, \boldsymbol{w}) \in \Sigma}
			\Big[
			\sum_{(i,j)\in\mathcal{E}} c_{ij}\, w_{ij}
			+ \tfrac{\rho}{2}\,\|\boldsymbol{z}_k - \boldsymbol{w} + \boldsymbol{\mu}_k\|^2
			+ \tfrac{\rho}{2}\,\|\boldsymbol{y}_k - \boldsymbol{u} + \boldsymbol{\eta}_k\|^2
			\Big],\label{cent-start}\\
			\boldsymbol{z}_{k+1}
			&= \arg\min_{\boldsymbol{z}\in\mathcal{Z}}
			\ \|\boldsymbol{z} - \boldsymbol{w}_{k+1} + \boldsymbol{\mu}_k\|^2, \label{Projj}\\
			\boldsymbol{y}_{k+1}
			&= \arg\min_{\boldsymbol{y}\in \{0,1\}^{2m|F|}}
			\ \|\boldsymbol{y} - \boldsymbol{u}_{k+1} + \boldsymbol{\eta}_k\|^2,\label{Projj2}\\
			\boldsymbol{\mu}_{k+1} &= \boldsymbol{\mu}_k + \boldsymbol{z}_{k+1} - \boldsymbol{w}_{k+1},\\
			\boldsymbol{\eta}_{k+1} &= \boldsymbol{\eta}_k + \boldsymbol{y}_{k+1} - \boldsymbol{u}_{k+1}.\label{cent-end}
		\end{align}
	\end{subequations}
	
	We define the constraint set \( \Sigma \) by:
	\[
	\Sigma = \left\{ (\boldsymbol{u}, \boldsymbol{w}) \in [0,1]^{2m|F|}\times[0,1]^m  : \text{constraints~\eqref{eq:constraint3}--\eqref{eq:constraint6} hold} \right\}.
	\]
	Note that constraint~\eqref{eq:constraint1} is excluded from the set $\Sigma$ 
	because, by Proposition~\ref{Proposition}, problem~\eqref{Projj} reduces to an MST problem with weights defined in~\eqref{weight} which guarantees that $\boldsymbol{z}_k$ induces a tree for all $k\geq1$. The second integer problem~\eqref{Projj2} is simply a straightforward component-wise 
	projection of $\boldsymbol{u}_{k+1} - \boldsymbol{\eta}_k$ onto 
	$\{0,1\}^{2m|F|}$.

	\subsection{The Distributed Algorithm}
	To decentralize problem \eqref{eq:MCF_obj}--\eqref{eq:MCF_binary}, we express the objective function and constraints as:
	\begin{equation}
		\sum_{(i,j) \in \mathcal{E}} c_{ij} w_{ij}
		= \frac{1}{2} \sum_{i \in \mathcal{V}} \left( \sum_{j \in N(i)} c_{ij} w_{ij} \right)
		=: \sum_{i \in \mathcal{V}} f^i(\boldsymbol{w}).
		\label{Q1}
	\end{equation}
	
	For the constraints in~\eqref{eq:constraint2}, we express them in terms of local variables:
	\begin{equation}
		\boldsymbol{w}^i = \boldsymbol{z}^i, 
		\quad (\boldsymbol{u}^f)^i = (\boldsymbol{y}^f)^i, 
		\quad i \in \mathcal{V}, \quad
		\boldsymbol{w}^i = \boldsymbol{w}^j, \quad 
		(\boldsymbol{u}^f)^i = (\boldsymbol{u}^f)^j, 
		\quad f \in F,\ (i,j) \in \mathcal{E}.
	\end{equation}
	
	The constraint~\eqref{eq:constraint3} can be decentralized for each agent \(i \in \mathcal{V}\) as:
	\begin{equation}
		\sum_{j \in N^{-}(i)} \left( u_{ji}^f \right)^i 
		- \sum_{j \in N^{+}(i)} \left( u_{ij}^f \right)^i =
		\begin{cases}
			1 & \text{if } i = O(f), \\
			0 & \text{if } i \notin \{O(f), D(f)\}, \\
			-1 & \text{if } i = D(f),
		\end{cases}
		\quad \forall i \in \mathcal{V},\ \forall f \in F.\label{12}
	\end{equation}
	
	For constraint~\eqref{eq:constraint4}, we decentralize it as:
	\begin{equation}
		\left( u_{ij}^f \right)^i + \left( u_{ji}^f \right)^i 
		\leq \left( w_{ij} \right)^i,
		\quad \forall j \in N(i),\ \forall i \in \mathcal{V},\ \forall f \in F.
	\end{equation}
	
	For the diameter constraint~\eqref{eq:constraint6},	we assume that for each agent \(l \in \mathcal{V}\), the following local constraint is satisfied:
	
	\begin{equation}\sum_{(i,j) \in \mathcal{A}} \left( u_{ij}^f\right)^l \leq d, \quad \forall f \in F,\ l\in \mathcal{V}.\label{QQ}
	\end{equation}
	Thus, the local constraint set for each agent \(i \in \mathcal{V}\) is defined as:
	\[
	\Sigma^i = \left\{
	(\boldsymbol{u}^i, \boldsymbol{w}^i) \in [0,1]^{2m|F|}\times[0,1]^m  :
	\text{constraints~\eqref{12}--\eqref{QQ} hold}
	\right\}.
	\]
	
	By adapting the distributed algorithm for this problem, we get the following ADMM updates:
	\begin{subequations}
		\begin{align}
			(\boldsymbol{u}^i, \boldsymbol{w}^i)_{k+1} &= 
			\underset{(\boldsymbol{w}^i, \boldsymbol{u}^i) \in \Sigma^i}{\arg \min} 
			\Bigg\{
			f^i(\boldsymbol{w}^i)
			+ \frac{\rho}{2} \left\| \boldsymbol{z}^i_k - \boldsymbol{w}^i + \boldsymbol{\mu}^i_k \right\|^2
			+ \frac{\rho}{2} \left\| \boldsymbol{y}_k^i - \boldsymbol{u}^i + \boldsymbol{\eta}_k^i \right\|^2 \notag\\
			&\quad + (\boldsymbol{\nu}_k^i)^{\top} \boldsymbol{u}^i
			+ \rho \sum_{j \in N(i)} \left\| \boldsymbol{u}^i - \frac{\boldsymbol{u}_k^i + \boldsymbol{u}_k^j}{2} \right\|^2 \label{dist-start} \\
			&\quad + (\boldsymbol{\xi}_k^i)^{\top} \boldsymbol{w}^i
			+ \rho \sum_{j \in N(i)} \left\| \boldsymbol{w}^i - \frac{\boldsymbol{w}_k^i + \boldsymbol{w}_k^j}{2} \right\|^2
			\Bigg\}, \notag\\
			\boldsymbol{z}_{k+1}^i &= 
			\underset{\boldsymbol{z}^i \in \mathcal{Z}}{\arg \min}
			\left\| \boldsymbol{z}^i - \boldsymbol{w}_{k+1}^i + \boldsymbol{\mu}_k^i \right\|^2, \\
			\boldsymbol{y}_{k+1}^i &= 
			\arg \min_{\boldsymbol{y}^i \in \{0,1\}^{2m|F|}}
			\left\| \boldsymbol{y}^i - \boldsymbol{u}_{k+1}^i + \boldsymbol{\eta}_k^i \right\|^2, \\
			\boldsymbol{\nu}_{k+1}^i &= 
			\boldsymbol{\nu}_k^i + \sum_{j \in N(i)} \left( \boldsymbol{u}_{k+1}^i - \boldsymbol{u}_{k+1}^j \right), \\
			\boldsymbol{\xi}_{k+1}^i &= 
			\boldsymbol{\xi}_k^i + \sum_{j \in N(i)} \left( \boldsymbol{w}_{k+1}^i - \boldsymbol{w}_{k+1}^j \right), \\
			\boldsymbol{\mu}_{k+1}^i &= 
			\boldsymbol{\mu}_k^i + \boldsymbol{z}_{k+1}^i - \boldsymbol{w}_{k+1}^i, \\
			\boldsymbol{\eta}_{k+1}^i &= 
			\boldsymbol{\eta}_k^i + \boldsymbol{y}_{k+1}^i - \boldsymbol{u}_{k+1}^i.\label{dist-end}
		\end{align}
	\end{subequations}

	\section{Numerical Evaluation}
	\label{sec:numerical}
	
	In this section, we evaluate both the centralized ADMM-based algorithm \eqref{cent-start}--\eqref{cent-end} and the distributed version \eqref{dist-start}--\eqref{dist-end}, comparing their solution quality and computational efficiency against results obtained using the commercial solver \textsf{Gurobi}.
	
	All numerical experiments were performed on a workstation equipped with an Intel\textsuperscript{\textregistered} Core\texttrademark~i7-8650U CPU (4~cores, 8~threads, base frequency 1.90\, GHz, turbo boost up to 4.20\, GHz) and 16\, GB of DDR4 RAM.
	
	To measure the deviation of the ADMM solution from the optimal one, we define the optimality gap as
	\begin{equation}
		\textrm{Gap} = \left( \frac{\textrm{ADMM Obj}}{\textrm{Gurobi Obj}} - 1 \right) \times 100\%,
		\label{gap}
	\end{equation}
	where \(\textrm{ADMM Obj}\) and \(\textrm{Gurobi Obj}\) denote the objective values obtained by the ADMM algorithms and \textsf{Gurobi}, respectively. A positive gap indicates that the ADMM result is suboptimal relative to \textsf{Gurobi}, while \(\textrm{Gap} = 0\) denotes an exact match.
	
	We evaluate the performance using both Erdős–Rényi random graphs~\cite{erdos1959erdos} and benchmark datasets from \cite{carvalho2023data}. All algorithms are initialized with $\boldsymbol{w}_0 = \boldsymbol{w}^i_0 = \mathbf{1}_m$ for all $i \in \mathcal{V}$. The ADMM solution quality is benchmarked against the global optimum obtained via \textsf{Gurobi}~\cite{gurobi}, with the error tolerance parameter set to $10^{-4}$.

	\subsection{Erd\H{o}s--R\'enyi Random Graphs}\label{Erdos}
	We first present numerical results on Erd\H{o}s--R\'enyi random graphs \cite{erdos1959erdos} with connectivity probabilities $p \in \{0.1, 0.5, 1\}$. We evaluate the algorithms across network sizes $n \in \{10, 50, 100, 200\}$ and penalty parameters $\rho \in \{0.1, 1, 10\}$. The positive cost vector $\boldsymbol{c}$ is chosen randomly by using the Matlab function \textsf{rand}. To provide a clearer view of the performances, all presented results are the average over $10$ independent executions of the algorithms for each configuration.
	\subsubsection*{Execution Time}
	
	Figure \ref{fig:Time} illustrates the execution time comparison. 
	
	\begin{itemize}
		\item \textbf{Small Scale ($n \le 50$):} For small networks, \textsf{Gurobi} is extremely fast, usually finishing in less than $0.1$\,s (see Figure \ref{fig:Time}(a)). The ADMM algorithm is slower here, taking between $1$s and $10$s because it needs extra time at the start to set up its mathematical variables (initialization overhead).
		
		\item \textbf{Large Scale ($n \ge 100$):} As the network size grows, the situation changes. \textsf{Gurobi}'s solve time increases, often taking over $100$\,s or failing to finish for $n=200$ especially for dense graph ($p=0.5$ or $p=1$). In contrast, ADMM handles large networks much better, keeping execution times around $10^2$\,s even for the largest cases (see Figure \ref{fig:Time}(f)).
		
		\item \textbf{Impact of $\rho$:} Choosing a higher penalty value ($\rho=10$) improves the convergence speed of the centralized algorithm in dense graphs (see Figure \ref{fig:Time}). This occurs because a larger penalty prioritizes feasibility, forcing the variables to satisfy the tree constraints more rapidly. However, this effect is less pronounced in the distributed case, where the algorithm must simultaneously balance local feasibility and neighbor consensus.
	\end{itemize}
	\subsubsection*{Optimality Gap}
	
	The optimality gaps are illustrated in Figure \ref{fig:gapcentr} (Centralized) and Figure \ref{fig:gapdist} (Distributed). The following observations can be made:
	
	\begin{itemize}
		\item \textbf{Connectivity Sensitivity:} The algorithms demonstrate high robustness in dense networks. Specifically, for fully connected graphs ($p=1.0$), the gap remains negligible (below $0.2\%$) regardless of the network size. Sparse graphs ($p=0.1$) proved more challenging; for instance, at $n=10$ with a penalty of $\rho=10$, the gap increases to approximately $3.0\%$ (Figure \ref{fig:gapcentr}(a)). However, choosing a smaller penalty ($\rho=0.1$) effectively reduces this gap to near zero.
		
		\item \textbf{Scalability:} The distributed algorithm maintains an average gap below $0.6\%$ for $p \ge 0.5$, even as the network size increases to $n=200$ (Figure \ref{fig:gapdist}(d)). This confirms that the proposed method scales efficiently and avoids the prohibitive computational burden of exact solvers without sacrificing solution quality.
	\end{itemize}
	
	\begin{figure}[htbp!]
		\centering
		\begin{minipage}[t]{0.32\textwidth}
			\centering
			\includegraphics[width=\linewidth]{\detokenize{Figures/p0.1rho0.1}}
			\\[2pt] {\small (a) $p=0.1, \ \rho=0.1$ \label{fig:Time_a}}
		\end{minipage}
		\hfill
		\begin{minipage}[t]{0.32\textwidth}
			\centering
			\includegraphics[width=\linewidth]{\detokenize{Figures/p0.1rho1}}
			\\[2pt] {\small (b) $p=0.1, \ \rho=1$ \label{fig:Time_b}}
		\end{minipage}
		\hfill
		\begin{minipage}[t]{0.32\textwidth}
			\centering
			\includegraphics[width=\linewidth]{\detokenize{Figures/p0.1rho10}}
			\\[2pt] {\small (c) $p=0.1, \ \rho=10$ \label{fig:Time_c}}
		\end{minipage}
		
		\vspace{1em} 
		
		\begin{minipage}[t]{0.32\textwidth}
			\centering
			\includegraphics[width=\linewidth]{\detokenize{Figures/p0.5rho0.1}}
			\\[2pt] {\small (d) $p=0.5, \ \rho=0.1$ \label{fig:Time_d}}
		\end{minipage}
		\hfill
		\begin{minipage}[t]{0.32\textwidth}
			\centering
			\includegraphics[width=\linewidth]{\detokenize{Figures/p0.5rho1}}
			\\[2pt] {\small (e) $p=0.5, \ \rho=1$ \label{fig:Time_e}}
		\end{minipage}
		\hfill
		\begin{minipage}[t]{0.32\textwidth}
			\centering
			\includegraphics[width=\linewidth]{\detokenize{Figures/p0.5rho10}}
			\\[2pt] {\small (f) $p=0.5, \ \rho=10$ \label{fig:Time_f}}
		\end{minipage}
		
		\vspace{1em} 
		
		\begin{minipage}[t]{0.32\textwidth}
			\centering
			\includegraphics[width=\linewidth]{\detokenize{Figures/p1rho0.1}}
			\\[2pt] {\small (g) $p=1, \ \rho=0.1$ \label{fig:Time_g}}
		\end{minipage}
		\hfill
		\begin{minipage}[t]{0.32\textwidth}
			\centering
			\includegraphics[width=\linewidth]{\detokenize{Figures/p1rho1}}
			\\[2pt] {\small (h) $p=1, \ \rho=1$ \label{fig:Time_h}}
		\end{minipage}
		\hfill
		\begin{minipage}[t]{0.32\textwidth}
			\centering
			\includegraphics[width=\linewidth]{\detokenize{Figures/p1rho10}}
			\\[2pt] {\small (i) $p=1, \ \rho=10$ \label{fig:Time_i}}
		\end{minipage}
		\caption{Average execution time and performance comparison of centralized and distributed ADMM with \textsf{Gurobi} for different values of $n$, $p$, and $\rho$.}
		\label{fig:Time}
	\end{figure}
	
	\begin{figure}[htbp!]
		\centering
		\begin{minipage}[t]{0.35\textwidth}
			\centering
			\includegraphics[width=\linewidth]{\detokenize{Figures/gapcent10}}
			\\[2pt] {\small (a) $n=10$}
		\end{minipage}\hspace{0.02\textwidth}
		\begin{minipage}[t]{0.35\textwidth}
			\centering
			\includegraphics[width=\linewidth]{\detokenize{Figures/gapcent50}}
			\\[2pt] {\small (b) $n=50$}
		\end{minipage}
		
		\vspace{0.5em}
		
		\begin{minipage}[t]{0.35\textwidth}
			\centering
			\includegraphics[width=\linewidth]{\detokenize{Figures/gapcent100}}
			\\[2pt] {\small (c) $n=100$}
		\end{minipage}\hspace{0.02\textwidth}
		\begin{minipage}[t]{0.35\textwidth}
			\centering
			\includegraphics[width=\linewidth]{\detokenize{Figures/gapcent200}}
			\\[2pt] {\small (d) $n=200$}
		\end{minipage}
		
		\caption{Average optimality gap for the centralized ADMM across various values of $n$, $p$, and $\rho$.}
		\label{fig:gapcentr}
	\end{figure}
	
	\begin{figure}[htbp!]
		\centering
		\begin{minipage}[t]{0.35\textwidth}
			\centering
			\includegraphics[width=\linewidth]{\detokenize{Figures/gapdis10}}
			\\[2pt] {\small (a) $n=10$}
		\end{minipage}\hspace{0.02\textwidth}
		\begin{minipage}[t]{0.35\textwidth}
			\centering
			\includegraphics[width=\linewidth]{\detokenize{Figures/gapdis50}}
			\\[2pt] {\small (b) $n=50$}
		\end{minipage}
		
		\vspace{0.5em}
		
		\begin{minipage}[t]{0.35\textwidth}
			\centering
			\includegraphics[width=\linewidth]{\detokenize{Figures/gapdis100}}
			\\[2pt] {\small (c) $n=100$}
		\end{minipage}\hspace{0.02\textwidth}
		\begin{minipage}[t]{0.35\textwidth}
			\centering
			\includegraphics[width=\linewidth]{\detokenize{Figures/gapdis200}}
			\\[2pt] {\small (d) $n=200$}
		\end{minipage}
		
		\caption{Average optimality gap for the distributed ADMM across various values of $n$, $p$, and $\rho$.}
		\label{fig:gapdist}
	\end{figure}
	
	\subsection{Benchmark Instances}\label{Benchmark}
	
	To evaluate the performance of the proposed framework on standardized data, we utilize benchmark instances from \cite{carvalho2023data}. We consider two specific instances and supplement them with origin-commodity sets to complete the problem formulation. The configurations and the global optima obtained via \textsf{Gurobi} are summarized in Table \ref{tab:benchmark_config}.
	
	\begin{table}[htb!]
		\centering
		\caption{Configuration of the benchmark instances and optimal values obtained via \textsf{Gurobi}.}
		\label{tab:benchmark_config}
		\small
		\begin{tabular}{lccccc}
			\hline
			Instance Name & $n$ & $m$ & Origins ($O$) & Commodities ($F$) & Optimal Obj. \\
			\hline
			\texttt{B-HCMST-c-10-1} & 10 & 45  & $\{7, 1\}$ & $\{3, 6\}$ & 52.25 \\
			\texttt{B-HCMST-e-25-20} & 25 & 300 & $\{21, 23, 4, 16\}$ & $\{3, 7, 14, 24, 25\}$ & 54.37 \\
			\hline
		\end{tabular}
	\end{table}
	
	\subsubsection*{Performance Analysis}
	
	The experimental results for these instances reveal the following insights:
	
	\begin{itemize}
		\item \textbf{Convergence Behavior:} Figures \ref{datafig:10-err} and \ref{datafig:n25-err} show that residuals in the centralized case diminish rapidly. In the distributed setting, however, high penalty values ($\rho=10$) can impede convergence, causing residuals to plateau or exhibit oscillatory behavior (see Figure \ref{datafig:n25-err}(c)).
		
		\item \textbf{Objective Optimization:} As shown in Figures \ref{datafig:n10-obj} and \ref{datafig:25-obj}, both algorithms asymptotically approach the \textsf{Gurobi} optimal solution. While the distributed method may exhibit an initial objective spike, it stabilizes at the optimal value within approximately 50 iterations, demonstrating robustness to parameter selection in dense topologies.
		
		\item \textbf{Agent Consensus:} The norm trajectories in Figures \ref{datafig:10-agent} and \ref{datafig:n25-agent} confirm successful synchronization. The perfectly overlapping trajectories indicate that all agents reach a unified state, effectively satisfying the distributed consensus constraints.
	\end{itemize}
	
	\begin{figure}[htbp!]
		\centering
		\begin{minipage}[t]{0.33\textwidth}
			\centering
			\includegraphics[width=\linewidth]{\detokenize{Figures/errorn10delta0.1}}
			\\[2pt] {\small (a) $\rho=0.1$}
		\end{minipage}\hfill
		\begin{minipage}[t]{0.33\textwidth}
			\centering
			\includegraphics[width=\linewidth]{\detokenize{Figures/errorn10delta1}}
			\\[2pt] {\small (b) $\rho=1$}
		\end{minipage}\hfill
		\begin{minipage}[t]{0.33\textwidth}
			\centering
			\includegraphics[width=\linewidth]{\detokenize{Figures/errorn10delta10}}
			\\[2pt] {\small (c) $\rho=10$}
		\end{minipage}
		\caption{Residual error evolution for centralized and distributed ADMM with \(n=10\) across \(\rho \in \{0.1,1,10\}\).}
		\label{datafig:10-err}
	\end{figure}
	
	\begin{figure}[htbp!]
		\centering
		\begin{minipage}[t]{0.33\textwidth}
			\centering
			\includegraphics[width=\linewidth]{\detokenize{Figures/objn10delta0.1}}
			\\[2pt] {\small (a) $\rho=0.1$}
		\end{minipage}\hfill
		\begin{minipage}[t]{0.33\textwidth}
			\centering
			\includegraphics[width=\linewidth]{\detokenize{Figures/objn10delta1}}
			\\[2pt] {\small (b) $\rho=1$}
		\end{minipage}\hfill
		\begin{minipage}[t]{0.33\textwidth}
			\centering
			\includegraphics[width=\linewidth]{\detokenize{Figures/objn10delta10}}
			\\[2pt] {\small (c) $\rho=10$}
		\end{minipage}
		\caption{Objective value evolution for centralized and distributed ADMM with \(n=10\).}
		\label{datafig:n10-obj}
	\end{figure}
	
	\begin{figure}[htbp!]
		\centering
		\begin{minipage}[t]{0.33\textwidth}
			\centering
			\includegraphics[width=\linewidth]{\detokenize{Figures/agentn10delta0.1}}
			\\[2pt] {\small (a) $\rho=0.1$}
		\end{minipage}\hfill
		\begin{minipage}[t]{0.33\textwidth}
			\centering
			\includegraphics[width=\linewidth]{\detokenize{Figures/agentn10delta1}}
			\\[2pt] {\small (b) $\rho=1$}
		\end{minipage}\hfill
		\begin{minipage}[t]{0.33\textwidth}
			\centering
			\includegraphics[width=\linewidth]{\detokenize{Figures/agentn10delta10}}
			\\[2pt] {\small (c) $\rho=10$}
		\end{minipage}
		\caption{Agent norm evolution for centralized and distributed ADMM with \(n=10\).}
		\label{datafig:10-agent}
	\end{figure}
	
	\begin{figure}[htbp!]
		\centering
		\begin{minipage}[t]{0.33\textwidth}
			\centering
			\includegraphics[width=\linewidth]{\detokenize{Figures/errorn25delta0.1}}
			\\[2pt] {\small (a) $\rho=0.1$}
		\end{minipage}\hfill
		\begin{minipage}[t]{0.33\textwidth}
			\centering
			\includegraphics[width=\linewidth]{\detokenize{Figures/errorn25delta1}}
			\\[2pt] {\small (b) $\rho=1$}
		\end{minipage}\hfill
		\begin{minipage}[t]{0.33\textwidth}
			\centering
			\includegraphics[width=\linewidth]{\detokenize{Figures/errorn25delta10}}
			\\[2pt] {\small (c) $\rho=10$}
		\end{minipage}
		\caption{Residual error evolution for centralized and distributed ADMM with \(n=25\).}
		\label{datafig:n25-err}
	\end{figure}
	
	\begin{figure}[htbp!]
		\centering
		\begin{minipage}[t]{0.33\textwidth}
			\centering
			\includegraphics[width=\linewidth]{\detokenize{Figures/objn25delta0.1}}
			\\[2pt] {\small (a) $\rho=0.1$}
		\end{minipage}\hfill
		\begin{minipage}[t]{0.33\textwidth}
			\centering
			\includegraphics[width=\linewidth]{\detokenize{Figures/objn25delta1}}
			\\[2pt] {\small (b) $\rho=1$}
		\end{minipage}\hfill
		\begin{minipage}[t]{0.33\textwidth}
			\centering
			\includegraphics[width=\linewidth]{\detokenize{Figures/objn25delta10}}
			\\[2pt] {\small (c) $\rho=10$}
		\end{minipage}
		\caption{Objective value evolution for centralized and distributed ADMM with \(n=25\).}
		\label{datafig:25-obj}
	\end{figure}
	
	\begin{figure}[htbp!]
		\centering
		\begin{minipage}[t]{0.33\textwidth}
			\centering
			\includegraphics[width=\linewidth]{\detokenize{Figures/agentn25delta0.1}}
			\\[2pt] {\small (a) $\rho=0.1$}
		\end{minipage}\hfill
		\begin{minipage}[t]{0.33\textwidth}
			\centering
			\includegraphics[width=\linewidth]{\detokenize{Figures/agentn25delta1}}
			\\[2pt] {\small (b) $\rho=1$}
		\end{minipage}\hfill
		\begin{minipage}[t]{0.33\textwidth}
			\centering
			\includegraphics[width=\linewidth]{\detokenize{Figures/agentn25delta10}}
			\\[2pt] {\small (c) $\rho=10$}
		\end{minipage}
		\caption{Agent's norm evolution for centralized and distributed ADMM with \(n=25\).}
		\label{datafig:n25-agent}
	\end{figure}
	
	\section{Extensions and Generalizations of the Framework}
	\label{sec:extensions}
	
	The proposed centralized and distributed ADMM frameworks rely on a fundamental
	structural decomposition: each iteration alternates between
	(i) solving a convex continuous subproblem and
	(ii) performing an exact projection onto the discrete topological set
	$\mathcal{Z}$ via a polynomial-time combinatorial subproblem.
	In the present work, this projection reduces to solving a MST or a MWRA problem with dynamically
	updated edge weights, ensuring strict feasibility with respect to the tree
	constraint at every iteration. This separation yields a modular \emph{plug-and-play} architecture in which the
	continuous physics and the discrete topology are decoupled, making the
	framework applicable to a broad class of mixed-integer nonconvex optimization
	problems beyond the hop-constrained setting.
	
	More generally, the only requirement imposed on the discrete update is the
	ability to efficiently solve a linear minimization problem over $\mathcal{Z}$.
	As a result, the framework can be interpreted as an instance of
	\emph{ADMM with exact combinatorial projections}, and it naturally admits
	extensions along several complementary dimensions.
	
	\subsection{Generalization of the Continuous Physics}
	The continuous sub-problem requires only that the objective function $f$ and the
	constraint function $g$ remain jointly convex in the continuous variables for
	any fixed topology $z \in \mathcal{Z}$.
	This mild assumption allows the integration of diverse operational and physical
	models without altering the combinatorial core of the algorithm.
	Representative application domains include:
	\begin{itemize}
		\item \emph{Power Distribution Networks:}
		Radial reconfiguration problems that minimize power losses, voltage
		deviations, or load imbalance under linearized AC or DC power flow models.
		In this setting, the MST or MWRA projection enforces the radiality constraint
		required for protection, reliability, and operational safety
		\cite{shen2019distributed,nejad2020scalable-bis,mokhtari2025alternating,mokhtari2024distributed}.
		
		\item \emph{Communication and Sensor Networks:}
		Energy-efficient, latency-aware, or interference-aware topology design
		problems in which tree structures guarantee loop-free routing, bounded
		control overhead, and simplified protocol design
		\cite{khan2003energy,smith2008dependency,martins2011augmented}.
		
		\item \emph{Prize-Collecting and Selective Network Design:}
		Models that associate penalties with unconnected nodes, yielding sparse
		tree-like topologies that balance infrastructure cost and coverage.
		Notable examples include prize-collecting Steiner tree formulations arising
		in telecommunication and sensor placement problems, which can be handled via
		exact or heuristic tree-based projection methods
		\cite{rehfeldt2018reduction,rehfeldt2020scip}.
	\end{itemize}
	
	\subsection{Modularity of the Combinatorial Projection Step}
	The discrete update steps correspond to projections onto $\mathcal{Z}$ and, in
	the present formulation, reduce to solving an MST or MWRA problem.
	Importantly, this projection step can be replaced or generalized without
	modifying the surrounding ADMM structure, provided that a suitable
	polynomial-time combinatorial algorithm or approximation procedure is
	available.
	Examples include:
	\begin{itemize}
		\item \emph{Minimum Spanning Forests:}
		Relaxing global connectivity allows the design of clustered or islanded
		network topologies, which are relevant in microgrids and resilient
		communication systems \cite{korte2018combinatorial}.
		
		\item \emph{Multi-Rooted Arborescences:}
		By introducing virtual super-roots, the framework extends naturally to
		multicast routing, multi-depot facility location, and hub-and-spoke network
		design problems.
		
		\item \emph{Approximate Projections for NP-Hard Topologies:}
		For constraints such as diameter bounds, degree limits, or survivability
		requirements, exact projection becomes NP-hard.
		In such cases, the discrete update can be replaced by approximation or
		heuristic algorithms, preserving the overall ADMM structure while trading
		exact feasibility for improved scalability.
	\end{itemize}
	
	\subsection{Algorithmic Refinements for Scalability and Robustness}
	The ADMM backbone further supports algorithmic refinements aimed at large-scale,
	distributed, or uncertain environments:
	\begin{itemize}
		\item \emph{Stochastic and Inexact Updates:}
		Approximate solutions of the continuous subproblem or stochastic gradient
		approximations reduce per-iteration complexity and improve robustness under
		uncertain or time-varying data, in line with stochastic and inexact ADMM
		theory \cite{bai2025inexact,boyd2011distributed}.
		
		\item \emph{Asynchronous and Privacy-Preserving Variants:}
		Asynchronous updates mitigate communication delays and packet losses, while
		localized updates enhance fault tolerance and scalability in decentralized
		implementations \cite{jin2025stochastic}.
	\end{itemize}
	
Despite the absence of global optimality guarantees due to nonconvexity, the framework enforces strict topological feasibility at every iteration, making it well suited for real-time and large-scale optimization of tree-constrained networked systems.

	\section*{Acknowledgments}
	
	This work was supported by the Office of Naval Research (ONR) under grant N00014-24-1-2095. The authors also wish to thank the anonymous reviewers for their insightful comments and constructive suggestions, which helped improve the quality and clarity of this manuscript.

	\section*{Appendix: Derivation of the Distributed ADMM}\label{Appendix}
	
	First, we derive the distributed version of the ADMM algorithm on directed	graphs. Then, we deduce the algorithm on undirected graphs by considering	both directions on the edges.
	
	Introduce the variables $\boldsymbol{t}%
	^{ij}, \boldsymbol{s}^{ij} \in \mathbb{R}^m $, where $(i, j) \in \mathcal{A} 
	$. Problem \eqref{distributed} is equivalent to:
	\begin{equation}
		\begin{array}{ll}
			\text{minimize} & \sum_{i \in \mathcal{V}} f^i(\boldsymbol{x}^i, \boldsymbol{%
				z}^i) \\ 
			\text{subject to} & \sum_{i \in \mathcal{V}} g^i(\boldsymbol{x}^i, 
			\boldsymbol{z}^i) \leq \boldsymbol{0}_q, \\ 
			& \boldsymbol{w}^i = \boldsymbol{z}^i, \\ 
			& \boldsymbol{t}^{ji} = \boldsymbol{t}^{ij} = \boldsymbol{x}^i, \quad 
			\boldsymbol{s}^{ji} = \boldsymbol{s}^{ij} = \boldsymbol{w}^i,
		\end{array}%
	\end{equation}
	and consider the Lagrangian:
	\begin{eqnarray*}
		&&\mathcal{L}_{\rho }(\boldsymbol{x},\boldsymbol{w},\boldsymbol{z},%
		\boldsymbol{\mu },\boldsymbol{t},\boldsymbol{s},\boldsymbol{\alpha },%
		\boldsymbol{\beta },\boldsymbol{\gamma },\boldsymbol{\delta }) \\
		&=&\sum_{i\in \mathcal{V}}f^{i}(\boldsymbol{x}^{i},\boldsymbol{w}%
		^{i})+\sum_{i\in \mathcal{V}}\boldsymbol{\mu }^{i\top }(\boldsymbol{z}^{i}-%
		\boldsymbol{w}^{i})+\frac{\rho }{2}\sum_{i\in \mathcal{V}}\left\Vert 
		\boldsymbol{z}^{i}-\boldsymbol{w}^{i}\right\Vert ^{2} \\
		&&+\sum_{(i,j)\in \mathcal{A}}(\boldsymbol{\alpha }^{ij})^{\top }(%
		\boldsymbol{x}^{i}-\boldsymbol{t}^{ij})+\frac{\rho }{2}\sum_{(i,j)\in 
			\mathcal{A}}\left\Vert \boldsymbol{x}^{i}-\boldsymbol{t}^{ij}\right\Vert^{2}+\sum_{(i,j)\in \mathcal{A}}(\boldsymbol{\beta }%
		^{ij})^{\top }(\boldsymbol{x}^{i}-\boldsymbol{t}^{ij}) \\
		&&+\frac{\rho }{2}\sum_{(i,j)\in \mathcal{A}}\left\Vert \boldsymbol{x}^{i}-%
		\boldsymbol{t}^{ji}\right\Vert ^{2}+\sum_{(i,j)\in \mathcal{%
				A}}(\boldsymbol{\gamma }^{ij})^{\top }(\boldsymbol{w}^{i}-\boldsymbol{s}%
		^{ij})+\frac{\rho }{2}\sum_{(i,j)\in \mathcal{A}}\left\Vert \boldsymbol{w}%
		^{i}-\boldsymbol{s}^{ij}\right\Vert ^{2} \\
		&&+\sum_{(i,j)\in \mathcal{A}}(\boldsymbol{\delta }^{ij})^{\top }(%
		\boldsymbol{w}^{i}-\boldsymbol{s}^{ji})+\frac{\rho }{2}\sum_{(i,j)\in 
			\mathcal{A}}\left\Vert \boldsymbol{w}^{i}-\boldsymbol{s}^{ji}\right\Vert ^{2}.
	\end{eqnarray*}	
	The Lagrangian can be written in separable form as:
	\begin{eqnarray*}
		&&\mathcal{L}_{\rho }(\boldsymbol{x},\boldsymbol{w},\boldsymbol{z},%
		\boldsymbol{\mu },\boldsymbol{t},\boldsymbol{s},\boldsymbol{\alpha },%
		\boldsymbol{\beta },\boldsymbol{\gamma },\boldsymbol{\delta }) \\
		&=&\sum_{i\in \mathcal{V}}\sum_{j\in \mathcal{N}^{+}(i)}\mathcal{L}_{\rho
		}^{i}(\boldsymbol{x}^{i},\boldsymbol{w}^{i},\boldsymbol{z}^{i},\boldsymbol{%
			\mu }^{i},\boldsymbol{t}^{ij},\boldsymbol{s}^{ij},\boldsymbol{\alpha }^{ij},%
		\boldsymbol{\beta }^{ij},\boldsymbol{\gamma }^{ij},\boldsymbol{\delta }^{ij})
		\\
		&=&\sum_{i\in \mathcal{V}}\sum_{j\in \mathcal{N}^{-}(i)}\mathcal{L}_{\rho
		}^{i}(\boldsymbol{x}^{i},\boldsymbol{w}^{i},\boldsymbol{z}^{i},\boldsymbol{%
			\mu }^{i},\boldsymbol{t}^{ij},\boldsymbol{s}^{ij},\boldsymbol{\alpha }^{ij},%
		\boldsymbol{\beta }^{ij},\boldsymbol{\gamma }^{ij},\boldsymbol{\delta }%
		^{ij}).
	\end{eqnarray*}
	where $\mathcal{L}_{\rho}^i(\boldsymbol{x}^i, \boldsymbol{w}^i, \boldsymbol{z%
	}^i, \boldsymbol{\mu}^i, \boldsymbol{t}^{ij}, \boldsymbol{s}^{ij}, 
	\boldsymbol{\alpha}^{ij}, \boldsymbol{\beta}^{ij}, \boldsymbol{\gamma}^{ij}, 
	\boldsymbol{\delta}^{ij}) $ is given by:
	\begin{eqnarray*}
		&&\mathcal{L}_{\rho }^{i}(\boldsymbol{x}^{i},\boldsymbol{w}^{i},\boldsymbol{z%
		}^{i},\boldsymbol{\mu }^{i},\boldsymbol{t}^{ij},\boldsymbol{s}^{ij},%
		\boldsymbol{\alpha }^{ij},\boldsymbol{\beta }^{ij},\boldsymbol{\gamma }^{ij},%
		\boldsymbol{\delta }^{ij}) \\
		&=&f^{i}(\boldsymbol{x}^{i},\boldsymbol{w}^{i})+\boldsymbol{\mu }^{i\top }(%
		\boldsymbol{z}^{i}-\boldsymbol{w}^{i})+\frac{\rho }{2}\left\Vert \boldsymbol{%
			z}^{i}-\boldsymbol{w}^{i}\right\Vert ^{2} \\
		&&+\sum_{j\in \mathcal{N}^{+}(i)}(\boldsymbol{\alpha }^{ij})^{\top }(%
		\boldsymbol{x}^{i}-\boldsymbol{t}^{ij})+\frac{\rho }{2}\sum_{j\in \mathcal{N}%
			^{+}(i)}\left\Vert \boldsymbol{x}^{i}-\boldsymbol{t}^{ij}\right\Vert^{2} \\
		&&+\sum_{j\in \mathcal{N}^{+}(i)}(\boldsymbol{\beta }^{ij})^{\top }(%
		\boldsymbol{x}^{j}-\boldsymbol{t}^{ij})+\frac{\rho }{2}\sum_{j\in \mathcal{N}%
			^{+}(i)}\left\Vert \boldsymbol{x}^{j}-\boldsymbol{t}^{ij}\right\Vert^{2} \\
		&&+\sum_{j\in \mathcal{N}^{+}(i)}(\boldsymbol{\gamma }^{ij})^{\top }(%
		\boldsymbol{w}^{i}-\boldsymbol{s}^{ij})+\frac{\rho }{2}\sum_{j\in \mathcal{N}%
			^{+}(i)}\left\Vert \boldsymbol{w}^{i}-\boldsymbol{s}^{ij}\right\Vert^{2} \\
		&&+\sum_{j\in \mathcal{N}^{+}(i)}(\boldsymbol{\delta }^{ij})^{\top }(%
		\boldsymbol{w}^{j}-\boldsymbol{s}^{ij})+\frac{\rho }{2}\sum_{j\in \mathcal{N}%
			^{+}(i)}\left\Vert \boldsymbol{w}^{i}-\boldsymbol{s}^{ij}\right\Vert^{2}.
	\end{eqnarray*}
	The ADMM iterations for all $i \in \mathcal{V} $ and $j \in \mathcal{N}^+(i) 
	$ are given by:
	\begin{subequations}
		\begin{eqnarray*}
			(\boldsymbol{x}^i, \boldsymbol{w}^i)_{k+1} &=& \underset{(\boldsymbol{x}^i, 
				\boldsymbol{w}^i) \in \Sigma^i}{\arg \min} \mathcal{L}_{\rho}^i(\boldsymbol{x%
			}^i, \boldsymbol{w}^i, \boldsymbol{z}_k^i, \boldsymbol{\mu}_k^i, \boldsymbol{%
				t}_k^{ij}, \boldsymbol{s}_k^{ij}, \boldsymbol{\alpha}_k^{ij}, \boldsymbol{%
				\beta}_k^{ij}, \boldsymbol{\gamma}_k^{ij}, \boldsymbol{\delta}_k^{ij}), \\
			(\boldsymbol{s}^{ij}, \boldsymbol{t}^{ij})_{k+1} &=& \arg \min \mathcal{L}%
			_{\rho}^i(\boldsymbol{x}_{k+1}^i, \boldsymbol{w}_{k+1}^i, \boldsymbol{z}%
			_k^i, \boldsymbol{\mu}_k^i, \boldsymbol{t}^{ij}, \boldsymbol{s}^{ij}, 
			\boldsymbol{\alpha}_k^{ij}, \boldsymbol{\beta}_k^{ij}, \boldsymbol{\gamma}%
			_k^{ij}, \boldsymbol{\delta}_k^{ij}), \\
			\boldsymbol{z}_{k+1}^i &=& \underset{\boldsymbol{z}^i \in \mathcal{Z}}{\arg
				\min} \left\Vert \boldsymbol{z}^i - \boldsymbol{w}_{k+1}^i + \boldsymbol{\mu}%
			_k^i \right\Vert_{\mathbb{R}^m}^2, \\
			\boldsymbol{\mu}_{k+1}^i &=& \boldsymbol{\mu}_k^i + \rho \left( \boldsymbol{z}_{k+1}^i - \boldsymbol{w}_{k+1}^i \right), \\
			\boldsymbol{\alpha}_{k+1}^{ij} &=& \boldsymbol{\alpha}_k^{ij} + \rho \left( 
			\boldsymbol{x}_{k+1}^i - \boldsymbol{t}_{k+1}^{ij} \right), \\
			\boldsymbol{\beta}_{k+1}^{ij} &=& \boldsymbol{\beta}_k^{ij} + \rho \left( 
			\boldsymbol{x}_{k+1}^j - \boldsymbol{t}_{k+1}^{ij} \right), \\
			\boldsymbol{\gamma}_{k+1}^{ij} &=& \boldsymbol{\gamma}_k^{ij} + \rho \left( 
			\boldsymbol{w}_{k+1}^i - \boldsymbol{s}_{k+1}^{ij} \right), \\
			\boldsymbol{\delta}_{k+1}^{ij} &=& \boldsymbol{\delta}_k^{ij} + \rho \left( 
			\boldsymbol{w}_{k+1}^j - \boldsymbol{s}_{k+1}^{ij} \right).
		\end{eqnarray*}
	\end{subequations}
	where the local constraints set $\Sigma^i$ is defined in \eqref{local-constraints}. Since the problem is unconstrained and the objective functions are convex,
	and the variables $\boldsymbol{s}^{ij} $ and $\boldsymbol{t}^{ij} $ are
	uncoupled, we can simply take the gradient to solve for the variables:
	\begin{eqnarray*}
		\nabla_{\boldsymbol{t}_{k+1}^{ij}} \mathcal{L}_{\rho}^i(\boldsymbol{x}%
		_{k+1}^i, \boldsymbol{w}_{k+1}^i, \boldsymbol{z}_{k+1}^i, \boldsymbol{\mu}%
		_k^i, \boldsymbol{t}^{ij}, \boldsymbol{s}_k^{ij}, \boldsymbol{\alpha}%
		_k^{ij}, \boldsymbol{\beta}_k^{ij}, \boldsymbol{\gamma}_k^{ij}, \boldsymbol{%
			\delta}_k^{ij}) &=& \boldsymbol{0}_m, \\
		\nabla_{\boldsymbol{s}_{k+1}^{ij}} \mathcal{L}_{\rho}^i(\boldsymbol{x}%
		_{k+1}^i, \boldsymbol{w}_{k+1}^i, \boldsymbol{z}_{k+1}^i, \boldsymbol{\mu}%
		_k^i, \boldsymbol{t}_k^{ij}, \boldsymbol{s}^{ij}, \boldsymbol{\alpha}%
		_k^{ij}, \boldsymbol{\beta}_k^{ij}, \boldsymbol{\gamma}_k^{ij}, \boldsymbol{%
			\delta}_k^{ij}) &=&  \boldsymbol{0}_m,
	\end{eqnarray*}
	which yield:
	\begin{equation*}
		- \boldsymbol{\alpha}_k^{ij}- \boldsymbol{\beta}_k^{ji} -
		\rho (\boldsymbol{x}_{k+1}^i + \boldsymbol{x}_{k+1}^j) + 2 \rho \boldsymbol{t%
		}_{k+1}^{ij} = \boldsymbol{0}_m,
	\end{equation*}
	
	\begin{equation*}
		- \boldsymbol{\gamma}_k^{ij} - \boldsymbol{\delta}_k^{ji}
		- \rho (\boldsymbol{w}_{k+1}^i + \boldsymbol{w}_{k+1}^j) + 2 \rho 
		\boldsymbol{s}_{k+1}^{ij} = \boldsymbol{0}_m,
	\end{equation*}
	which gives:	
	\begin{subequations}
		\begin{eqnarray}
			\boldsymbol{t}_{k+1}^{ij} &=& \frac{1}{2 \rho} \boldsymbol{\alpha}_k^{ij} + \frac{1}{2 \rho} \boldsymbol{\beta}_k^{ij} + 
			\frac{1}{2} \left( \boldsymbol{x}_{k+1}^i + \boldsymbol{x}_{k+1}^j \right), \label{sub1} \\
			\boldsymbol{s}_{k+1}^{ij} &=& \frac{1}{2 \rho} \boldsymbol{\gamma}_k^{ij} + \frac{1}{2 \rho} \boldsymbol{\delta}_k^{ij} + 
			\frac{1}{2} \left( \boldsymbol{w}_{k+1}^i + \boldsymbol{w}_{k+1}^j \right). \label{sub2}
		\end{eqnarray}
	\end{subequations}
	
	From the adjoint equations, we get:
	\begin{eqnarray*}
		\boldsymbol{\alpha}_{k+1}^{ij} &=& \frac{1}{2} (\boldsymbol{\alpha}_k^{ij} - 
		\boldsymbol{\beta}_k^{ij}) + \frac{\rho}{2} (\boldsymbol{x}_{k+1}^i - 
		\boldsymbol{x}_{k+1}^j), \\
		\boldsymbol{\beta}_{k+1}^{ij} &=& \frac{1}{2} (\boldsymbol{\beta}_k^{ij} - 
		\boldsymbol{\alpha}_k^{ij}) + \frac{\rho}{2} (\boldsymbol{x}_{k+1}^j - 
		\boldsymbol{x}_{k+1}^i), \\
		\boldsymbol{\gamma}_{k+1}^{ij} &=& \frac{1}{2} (\boldsymbol{\gamma}_k^{ij} - 
		\boldsymbol{\delta}_k^{ji}) + \frac{\rho}{2} (\boldsymbol{w}_{k+1}^i - 
		\boldsymbol{w}_{k+1}^j), \\
		\boldsymbol{\delta}_{k+1}^{ij} &=& \frac{1}{2} (\boldsymbol{\delta}_k^{ij} - 
		\boldsymbol{\gamma}_k^{ij}) + \frac{\rho}{2} (\boldsymbol{w}_{k+1}^j - 
		\boldsymbol{w}_{k+1}^i).
	\end{eqnarray*}
	Summing up yields:
	\begin{equation*}
		\boldsymbol{\alpha}_{k+1}^{ij} + \boldsymbol{\beta}_{k+1}^{ij} = \boldsymbol{%
			0}_m, \quad \boldsymbol{\gamma}_{k+1}^{ij} + \boldsymbol{\delta}_{k+1}^{ij}
		= \boldsymbol{0}_m, \quad \forall k \geq 0.
	\end{equation*}
	Taking $\boldsymbol{\alpha}_0^{ij} = \boldsymbol{\beta}_0^{ij} = \boldsymbol{%
		\gamma}_0^{ij} = \boldsymbol{\delta}_0^{ij} = \boldsymbol{0}_q $, we obtain:
	\begin{equation*}
		\boldsymbol{\alpha}_k^{ij} + \boldsymbol{\beta}_k^{ij} = \boldsymbol{0}_q,
		\quad \boldsymbol{\gamma}_k^{ij} + \boldsymbol{\delta}_k^{ij} = \boldsymbol{0%
		}_m, \quad \forall k \geq 0.
	\end{equation*}
	This yields from \eqref{sub1}-\eqref{sub2}:
	\begin{eqnarray*}
		\boldsymbol{t}_{k+1}^{ij} &=& \frac{1}{2} (\boldsymbol{x}_{k+1}^i + 
		\boldsymbol{x}_{k+1}^j), \\
		\boldsymbol{s}_{k+1}^{ij} &=& \frac{1}{2} (\boldsymbol{w}_{k+1}^i + 
		\boldsymbol{w}_{k+1}^j).
	\end{eqnarray*}
	Thus, the dual variables become:
	\begin{subequations}
		\begin{eqnarray}
			\boldsymbol{\alpha}_{k+1}^{ij} &=& \boldsymbol{\alpha}_k^{ij} + \frac{\rho}{2} \left( \boldsymbol{x}_{k+1}^i - \boldsymbol{x}_{k+1}^j \right), \label{B1}\\
			\boldsymbol{\beta}_{k+1}^{ij}  &=& \boldsymbol{\beta}_k^{ij} + \frac{\rho}{2} \left( \boldsymbol{x}_{k+1}^j - \boldsymbol{x}_{k+1}^i \right), \\
			\boldsymbol{\gamma}_{k+1}^{ij} &=& \boldsymbol{\gamma}_k^{ij} +  \frac{\rho}{2} \left( \boldsymbol{w}_{k+1}^i - \boldsymbol{w}_{k+1}^j \right), \\
			\boldsymbol{\delta}_{k+1}^{ij} &=& \boldsymbol{\delta}_k^{ij} +  \frac{\rho}{2} \left( \boldsymbol{w}_{k+1}^j - \boldsymbol{w}_{k+1}^i \right).\label{B4}
		\end{eqnarray}
	\end{subequations}
	To handle the terms in the local Lagrangian $\mathcal{L}_{\rho}^i $, we
	eliminate the terms depending on the dual variables from the iterate ($%
	\boldsymbol{x}^i, \boldsymbol{w}^i $).
	
	Defining the new multipliers $\boldsymbol{\nu}_k^i$ and $\boldsymbol{\xi}_k^i $
	by:
	\begin{subequations}
		\begin{eqnarray}
			\boldsymbol{\nu}_k^i &=& \sum_{j \in \mathcal{N}^+(i)} \boldsymbol{\alpha}_k^{ij} 
			+ \sum_{j \in \mathcal{N}^-(i)} \boldsymbol{\beta}_k^{ji},\label{C1} \\
			\boldsymbol{\xi}_k^i &=& \sum_{j \in \mathcal{N}^+(i)} \boldsymbol{\gamma}_k^{ij} 
			+ \sum_{j \in \mathcal{N}^-(i)} \boldsymbol{\delta}_k^{ji}.\label{C2}
		\end{eqnarray}
	\end{subequations}
	
	By combining \eqref{B1}-\eqref{B4} and \eqref{C1}-\eqref{C2}, we get:
	
	\begin{eqnarray*}
		\boldsymbol{\nu}_{k+1}^i &=& \boldsymbol{\nu}_k^i + \frac{\rho}{2} \sum_{j
			\in N^-(i) \cup N^+(i)} \left( \boldsymbol{x}_{k+1}^i - \boldsymbol{x}%
		_{k+1}^j \right), \\
		\boldsymbol{\xi}_{k+1}^i &=& \boldsymbol{\xi}_k^i + \frac{\rho}{2} \sum_{j
			\in N^-(i) \cup N^+(i)} \left( \boldsymbol{w}_{k+1}^i - \boldsymbol{w}%
		_{k+1}^j \right).
	\end{eqnarray*}
	
	We conclude that:
	\begin{eqnarray}
		(\boldsymbol{x}^i, \boldsymbol{w}^i)_{k+1} &=& \underset{(\boldsymbol{x}^i, \boldsymbol{w}^i) \in \Sigma^i}{\arg \min} \left\{ f^i(\boldsymbol{x}^i, \boldsymbol{w}^i) + \boldsymbol{\mu}^{i\top} (\boldsymbol{z}^i - \boldsymbol{w}^i) \right. \nonumber \\
		&& \quad + \frac{\rho}{2} \left\Vert \boldsymbol{z}^i - \boldsymbol{w}^i \right\Vert_{\mathbb{R}^m}^2 + (\boldsymbol{\nu}_k^i)^{\top} \boldsymbol{x}^i \nonumber \\
		&& \quad + \frac{\rho}{2} \sum_{j \in N^-(i) \cup N^+(i)} \left\Vert \boldsymbol{x}^i - \frac{\boldsymbol{x}_k^i + \boldsymbol{x}_k^j}{2} \right\Vert^2 \nonumber \\
		&& \quad + (\boldsymbol{\xi}_k^i)^{\top} \boldsymbol{w}^i + \frac{\rho}{2} \sum_{j \in N^-(i) \cup N^+(i)} \left\Vert \boldsymbol{w}^i - \frac{\boldsymbol{w}_k^i + \boldsymbol{w}_k^j}{2} \right\Vert^2 \left. \right\}.
	\end{eqnarray}By dividing all dual variables by $\rho$,
	$$
	\tilde{\boldsymbol{\mu}}^{\,i} = \frac{\boldsymbol{\mu}^i}{\rho}, \quad 
	\tilde{\boldsymbol{\nu}}^{\,i} = \frac{\boldsymbol{\nu}^i}{\rho}, \quad 
	\tilde{\boldsymbol{\xi}}^{\,i} = \frac{\boldsymbol{\xi}^i}{\rho},
	$$
	and substituting these rescaled forms into the primal iterations yields the iterates~\eqref{distributed1}-\eqref{distributed4}.

	\bibliographystyle{IEEEtran}
	\bibliography{biblio}

\end{document}